\documentclass[10pt]{article}

\usepackage[english]{babel}
\usepackage{amsmath}
\usepackage{placeins}
\usepackage{amssymb}
\usepackage{latexsym}
\usepackage{amsmath}
\usepackage{amsthm}
\usepackage{enumerate}

\newtheorem{theorem}{Theorem}
\newtheorem{predl}{Theorem}
\newtheorem{lemma}{Lemma}

\newtheorem{rem}{Remark}

\begin{document}

\title{A Generalization of the Petrov Strong \\ Law of Large Numbers}


\author{Valery Korchevsky\thanks{Saint-Petersburg State University of Aerospace Instrumentation, Saint-Petersburg. \endgraf E-mail: \texttt{valery.korchevsky@gmail.com} }}

\date{}

\maketitle

\begin{abstract}
In 1969 V.V.~Petrov found a new sufficient condition for the applicability of the strong law of large numbers to sequences of independent random variables. He proved the following theorem: let $\{X_{n}\}_{n=1}^{\infty}$ be a sequence of independent random variables with finite variances and let $S_{n}=\sum_{k=1}^{n} X_{k}$. If $Var (S_{n})=O (n^{2}/\psi(n))$ for a positive non-decreasing function $\psi(x)$ such that $\sum 1/(n \psi(n)) < \infty$ (Petrov's condition) then the relation $(S_{n}-ES_{n})/n \to 0$ a.s. holds.

In 2008 V.V.~Petrov showed that under some additional assumptions Petrov's condition remains sufficient for the applicability of the strong law of large numbers to sequences of random variables without the independence condition.

In the present work, we generalize Petrov's results (for both dependent and independent random variables), using an arbitrary norming sequence in place of the classical normalization.
\end{abstract}

\bigskip

\noindent \textbf{Keywords:} strong law of large numbers, sequences of independent random variables, dependent random variables.
{\sloppy

}

\bigskip

\bigskip

\noindent {\Large{\textbf{1. Introduction}}}

\medskip

\noindent Following~\cite{Petr75}, we denote by $\Psi_c$ (or, respectively, $\Psi_d$) the set of  functions $\psi (x)$ such that $\psi (x)$ is positive and non-decreasing in the interval $x > x_0$ for some $x_0$ and the series $\sum \frac{1}{n \psi (n)}$ converges (respectively, diverges). The value $x_0$ is not assumed to be the same for different functions $\psi$. Examples of functions of the class $\Psi_{c}$ are the functions $x^{\delta}$ and $(\log x)^{1 + \delta}$ for any $\delta > 0$. The functions $\log x$ and $\log \log x$ belong to the class $\Psi_{d}$.
{\sloppy

}

The next result is classical Kolmogorov's theorem:

\begin{predl}\label{P101}
Let $\{X_{n}\}_{n=1}^{\infty}$ be a sequence of independent random variables with finite variances $Var (X_{n})$ and let $S_{n}=\sum_{k=1}^{n} X_{k}$. If
{\sloppy

}

\begin{equation}\label{e101}
    \sum_{n=1}^\infty \frac{Var (X_{n})}{n^{2}} < \infty
\end{equation}
then

\begin{equation}\label{e102}
    \frac{S_{n} - ES_{n}}{n} \to 0 \qquad \mbox{a.s.}
\end{equation}

\end{predl}

Another sufficient condition for the applicability of the strong law of large numbers to sequences of independent random variables was founded by Petrov~\cite{Petr69} (see also~\cite{Petr75}).

\begin{predl}[Petrov]\label{P102}
Let $\{X_{n}\}_{n=1}^{\infty}$ be a sequence of independent random variables with finite variances. If
{\sloppy

}

\begin{equation}\label{e103}
Var (S_{n}) = O \biggl( \frac{n^{2}}{\psi(n)} \biggr) \qquad
      \mbox{for some function } \psi \in \Psi_{c},
\end{equation}
then relation~\eqref{e102} holds.

\end{predl}

Relation~\eqref{e103} will be called Petrov's condition. It is known~\cite{Petr69} (see also~\cite{Egorov72}) that condition~\eqref{e103} in Theorem~\ref{P102} is optimal in the  following sense: it is impossible to replace condition~\eqref{e103} by the weaker assumption that corresponds to the replacement of $\psi \in \Psi_{c}$ by some function $\psi \in \Psi_{d}$.

If the random variables $X_1, X_2, \ldots$ are independent, then Petrov's condition is equivalent to the requirement that

\begin{equation}\label{e104}
      \sum_{k=1}^{n} Var (X_{k}) = O \biggl( \frac{n^{2}}{\psi (n)} \biggr) \qquad
      \mbox{for some function } \psi \in \Psi_{c}.
\end{equation}

It is proved (\cite[Theorem~1]{Ko10}) that~\eqref{e104} implies~\eqref{e101}. It follows that theorem~\ref{P102} is a consequence of Kolmogorov's theorem (Theorem~\ref{P101}). Nevertheless,  Petrov proved~\cite{Petr08, Petr09} that under some additional assumptions Petrov's condition is sufficient for the applicability of the strong law of large numbers to sequences of random variables without any independence assumptions.

\begin{predl}[Petrov~\cite{Petr08}]\label{P103}
Let $\{X_{n}\}_{n=1}^{\infty}$ be a sequence of non-negative random variables with finite variances. Suppose that conditions~\eqref{e103} is satisfied and
{\sloppy

}

\begin{equation}\label{e105}
     E(S_n -  S_m) \leqslant C (n-m) \qquad \mbox{for all sufficiently large } \; n - m,
\end{equation}
where $C$ is a constant. Then relation~\eqref{e102} holds.

\end{predl}

It is proved in~\cite{KoPetr10} the next generalization of Theorem~\ref{P103}:

\begin{predl}[Petrov and Korchevsky]\label{P104}

Let $\{X_{n}\}_{n=1}^{\infty}$ be a sequence of non-negative random variables with finite absolute moments of some order $p \geqslant 1$. Suppose that condition~\eqref{e105} is satisfied and

\begin{equation*}\label{e106}
      E|S_{n} - ES_{n}|^p = O \biggl( \frac{n^{p}}{\psi (n)} \biggr) \qquad
      \mbox{for some function } \psi \in \Psi_{c}.
\end{equation*}
Then relation~\eqref{e102} holds.

\end{predl}

\noindent (Theorem~\ref{P103} corresponds to the case $p = 2$).

The aim of present work is to generalize Theorems~\ref{P102} and~\ref{P104} using an arbitrary norming sequence in place of the classical normalization. Also we present a generalization of Theorem~1 in~\cite{Ko10}.

To prove the theorems of the work we use methods developed by Petrov~\cite{Petr08, Petr09}, Chandra and Goswami~\cite{ChGos92}, and Cs\"{o}rg\H{o}, Tandori, and Totik~\cite{CsTanTot83}.

\newpage

\noindent {\Large{\textbf{2. Main results}}}

\medskip

\begin{theorem}\label{T201}

Let $\{X_{n}\}_{n=1}^{\infty}$ be a sequence of non-negative random variables with finite absolute moments of some order $p \geqslant 1$. Assume that $\{a_{n}\}_{n=1}^{\infty}$ is non-decreasing unbounded sequence of positive numbers. If

\begin{equation}\label{e201}
ES_{n} = O (a_{n})
\end{equation}
and

\begin{equation}\label{e202}
      E|S_{n} - ES_{n}|^{p} = O \biggl( \frac{a_{n}^{p}}{\psi (a_{n})} \biggr) \qquad
      \mbox{for some function } \psi \in \Psi_{c},
\end{equation}
then
\begin{equation}\label{e203}
      \frac{S_{n} - ES_{n}}{a_{n}} \to 0 \qquad \mbox{a.s.}
\end{equation}

\end{theorem}

Theorem~\ref{T201} generalizes Theorem~\ref{P104}, which corresponds to the case $a_{n} = n$ for all ${n \geqslant 1}$. Moreover, in the case $a_{n} = n$ for all ${n \geqslant 1}$, condition~\eqref{e201} is less restrictive than assumption~\eqref{e105}.
{\sloppy

}

Let us indicate two consequences of Theorem~\ref{T201}.

\begin{theorem}\label{T202}
Let $\{X_{n}\}_{n=1}^{\infty}$ be a sequence of non-negative random variables with finite variances. Assume that $\{a_{n}\}_{n=1}^{\infty}$ is non-decreasing unbounded sequence of positive numbers. If condition~\eqref{e201} is satisfied and

\begin{equation*}\label{e204}
Var (S_{n}) = O \biggl( \frac{a_{n}^{2}}{\psi (a_{n})} \biggr) \qquad
      \mbox{for some function } \psi \in \Psi_{c},
\end{equation*}
then relation~\eqref{e203} holds.

\end{theorem}

We arrive at this proposition putting $p = 2$ in Theorem~\ref{T201}.

\begin{theorem}\label{T203}
Let $\{X_{n}\}_{n=1}^{\infty}$ be a sequence of non-negative random variables with finite absolute moments of some order $p \geqslant 1$. Assume that $\{w_{n}\}_{n=1}^{\infty}$ is a sequence of positive numbers,

\begin{equation*}\label{e205}
W_{n} = \sum_{k=1}^{n} w_{k}, \qquad T_{n} = \sum_{k=1}^{n} w_{k} X_{k}.
\end{equation*}
Suppose that $W_{n} \to \infty$ $(n \to \infty)$,

\begin{equation}\label{e206}
ET_{n} = O (W_{n}),
\end{equation}

\begin{equation*}\label{e207}
      E|T_{n} - ET_{n}|^{p} = O \biggl( \frac{W_{n}^{p}}{\psi (W_{n})} \biggr) \qquad
      \mbox{for some function } \psi \in \Psi_{c}.
\end{equation*}
Then

\begin{equation*}\label{e208}
      \frac{T_{n} - ET_{n}}{W_{n}} \to 0 \qquad \mbox{a.s.}
\end{equation*}

\end{theorem}

Theorem~\ref{T203} is a generalization of Theorem~1 in~\cite{KoPetr10} which includes condition

\begin{equation*}\label{e209}
       \sum_{k=m}^{n} w_{k} EX_{k} \leqslant C \sum_{k=m}^{n} w_{k} \qquad \mbox{for all sufficiently large } \; n - m,
\end{equation*}

\noindent instead assumption~\eqref{e206}. To prove Theorem~\ref{T203} we can put $a_{n} = W_{n}$, $Y_{n} = w_{n} X_{n}$ for all ${n \geqslant 1}$ and apply Theorem~\ref{T201} to the sequence of random variables $\{Y_{n}\}_{n=1}^{\infty}$.
{\sloppy

}

The next theorem generalizes Theorem~\ref{P102}, which corresponds to the case $a_{n} = n$ for all $n \geqslant 1$.

\begin{theorem}\label{T204}

Let $\{X_{n}\}_{n=1}^{\infty}$ be a sequence of independent random variables with finite variances. Assume that $\{a_{n}\}_{n=1}^{\infty}$ is non-decreasing unbounded sequence of positive numbers such that 

\begin{equation}\label{e210}
    \frac{a_{2n}}{a_{n}} \leqslant Q \qquad \mbox{for all sufficiently large } \; n,
\end{equation}
where $Q$ is a constant. If

\begin{equation}\label{e211}
      Var (S_{n}) = O \biggl( \frac{a_{n}^{2}}{\psi (n)} \biggr) \qquad
      \mbox{for some function } \psi \in \Psi_{c},
\end{equation}
then relation~\eqref{e203} holds.

\end{theorem}

\begin{rem}\label{Rem201}
We cannot omit condition~\eqref{e210} in Theorem~\ref{T204} (See Example~1 below).
{\sloppy

}

\end{rem}

As mentioned above, in~\cite{Ko10} was proved that condition~\eqref{e104} implies~\eqref{e101}. The next theorem generalizes this result.

\begin{theorem}\label{T205}

Let $\{b_{n}\}_{n=1}^{\infty}$ be a sequence of non-negative numbers. Assume that $\{a_{n}\}_{n=1}^{\infty}$ is non-decreasing unbounded sequence of positive numbers such that condition~\eqref{e210} is satisfied. If

\begin{equation}\label{e212}
      \sum_{k=1}^{n} b_{k} = O \biggl( \frac{a_{n}^{2}}{\psi (n)} \biggr) \qquad
      \mbox{for some function } \psi \in \Psi_{c},
\end{equation}
then

\begin{equation}\label{e213}
\sum_{n=1}^{\infty} \frac{b_{n}}{a_{n}^{2}} < \infty.
\end{equation}

\end{theorem}

\begin{rem}\label{Rem202}

We cannot omit condition~\eqref{e210} in Theorem~\ref{T205}.
{\sloppy

}

\end{rem}

\noindent Indeed, let $b_{1} = b_{2} = 1$, $b_{n} = 2^{n}/n - 2^{n-1}/(n-1)$ for all $n \geqslant 3$. Then

\begin{equation*}\label{e214}
\sum_{k=1}^{n} b_{k} = \frac{2^{n}}{n} \qquad \mbox{for all } \; n \geqslant 3.
\end{equation*}

\noindent Thus, the sequence $\{b_{n}\}_{n=1}^{\infty}$ satisfies condition~\eqref{e212} with $a_{n} = 2^{n/2}$, $n \geqslant 1$ and function $\psi (x) = x$, $x > 0$ (belonging to~$\Psi_{c}$). But relation~\eqref{e213} does not hold since
{\sloppy

}

\begin{equation*}\label{e215}
\sum_{n=3}^{\infty} \frac{b_{n}}{a_{n}^{2}} = \sum_{n=3}^{\infty} \frac{2^{n}/n - 2^{n-1}/(n-1)}{2^{n}} = \sum_{n=3}^{\infty} \frac{n-2}{2 n (n-1)} = \infty.
\end{equation*}

\bigskip

\noindent {\Large{\textbf{3. Proofs}}}

\medskip

\noindent To prove Theorems~\ref{T201} and~\ref{T204} we need the following proposition.

\begin{lemma}[see~\cite{Petr08}]\label{Lem301} If $\psi(x) \in \Psi_{c}$, then the series $\sum 1/\psi(b^{n})$ converges for every $b>1$.
{\sloppy

}

\end{lemma}

\begin{proof}[Proof of Theorem~\ref{T201}]

\noindent By assumption~\eqref{e201} there is a constant $A$ such that inequality

\begin{equation*}\label{e301}
 ES_{n}/a_{n} \leqslant A
\end{equation*}

\noindent is satisfied for each $n \geqslant 1$. Let $\alpha > 1$, $\varepsilon > 0$ and $L = [A/\varepsilon]$, the integer part of $A/\varepsilon$. Put
{\sloppy

}

\begin{equation*}\label{e302}
      m_{1} = \inf \{ m \geqslant 0 : \alpha^{m} \leqslant a_{n} < \alpha^{m+1} \; \mbox{for some } n \},
\end{equation*}

\begin{equation*}\label{e303}
      m_{l} = \inf \{ m > m_{l-1} : \alpha^{m} \leqslant a_{n} < \alpha^{m+1} \; \mbox{for some } n \} \qquad \mbox{ for } l \geqslant 2.
\end{equation*}

\noindent We recall that $a_{n} \uparrow \infty$, so $\{m_{l}\}_{l=1}^{\infty}$ is a subsequence of integers satisfying $0 \leqslant m_{1} < m_{2} < \ldots \;$ and $\; m_{l} \to \infty$ $(l \to \infty)$. For each pair of integers $l$ and $s$ such that $l \geqslant 1$, $s = 0,1,\ldots,L$, put
{\sloppy

}

\begin{equation*}\label{e304}
       A_{s} (l) = \{ k : \alpha^{m_{l}} \leqslant a_{k} < \alpha^{m_{l}+1}, \; \frac{ES_{k}}{a_{k}} \in [s \varepsilon, (s+1)\varepsilon) \}.
\end{equation*}

\noindent Let $k^{-}_{s} (l) = \inf A_{s} (l)$, $k^{+}_{s} (l) = \sup A_{s} (l)$, if the set $A_{s} (l)$ is not empty, and let $k^{-}_{s} (l) = k^{+}_{s} (l) = \inf \{ k : \alpha^{m_{l}} \leqslant a_{k} < \alpha^{m_{l}+1}\}$ otherwise.
{\sloppy

}

By the definition of $k^{\pm}_{s} (l)$ for any $l \geqslant 1$ and $s = 0,1,\ldots,L$ we have

\begin{equation*}\label{e305}
       a_{k^{\pm}_{s} (l)} \geqslant \alpha^{m_{l}}.
\end{equation*}

\noindent Hence, using assumption~\eqref{e202} and Lemma~\ref{Lem301}, by Chebyshev's inequality for any $s = 0,1,\ldots,L$ and $\lambda > 0$ we obtain

\begin{multline*}
\sum_{l=1}^{\infty} P \left( \left| \frac{S_{k^{\pm}_{s} (l)} - ES_{k^{\pm}_{s} (l)}}{a_{k^{\pm}_{s} (l)}} \right| > \lambda \right) \leqslant \frac{1}{\lambda^{p}} \sum_{l=1}^{\infty} \frac{E|S_{k^{\pm}_{s} (l)} - ES_{k^{\pm}_{s} (l)}|^{p}}{(a_{k^{\pm}_{s} (l)})^{p}} \leqslant \\ \leqslant C \lambda^{-p} \frac{1}{\psi (a_{k^{\pm}_{s} (l)})} \leqslant C \lambda^{-p} \frac{1}{\psi (\alpha^{m_{l}})} < \infty.
\end{multline*}

\noindent The application of Borel--Cantelli lemma yields to

\begin{equation}\label{e307}
       \frac{S_{k^{\pm}_{s} (l)} - ES_{k^{\pm}_{s} (l)}}{a_{k^{\pm}_{s} (l)}} \to 0 \qquad \mbox{a.s.} \qquad (l \to \infty)
\end{equation}
for any $s = 0,1,\ldots,L$.

Now for any natural number $n$ there exists $l=l(n)$ and $s=s(n)$, $\lim_{n \to \infty} l(n) = \infty$, $0 \leqslant s(n) \leqslant L$ such that

\begin{equation*}\label{e308}
       \alpha^{m_{l}} \leqslant a_{n} < \alpha^{m_{l}+1}, \qquad \frac{ES_{n}}{a_{n}} \in [s \varepsilon, (s+1)\varepsilon).
\end{equation*}

\noindent By the definition of $k^{\pm}_{s} (l)$ we have

\begin{equation*}\label{e309}
k^{-}_{s} (l) \leqslant n \leqslant k^{+}_{s} (l), \qquad \left| \frac{ES_{k^{\pm}_{s} (l)}}{a_{k^{\pm}_{s} (l)}} - \frac{ES_{n}}{a_{n}} \right| < \varepsilon,
\end{equation*}

\noindent and so

\begin{multline*}
-\varepsilon - \left( 1 - \frac{1}{\alpha} \right) A + \frac{1}{\alpha} \frac{1}{a_{k^{-}_{s} (l)}} (S_{k^{-}_{s} (l)} - ES_{k^{-}_{s} (l)}) \leqslant \\ \leqslant -\varepsilon - \left( 1 - \frac{1}{\alpha} \right) \frac{ES_{k^{-}_{s} (l)}}{a_{k^{-}_{s} (l)}} + \frac{1}{\alpha} \frac{1}{a_{k^{-}_{s} (l)}} (S_{k^{-}_{s} (l)} - ES_{k^{-}_{s} (l)}) = \\ = -\varepsilon - \frac{ES_{k^{-}_{s} (l)}}{a_{k^{-}_{s} (l)}} + \frac{1}{\alpha} \frac{S_{k^{-}_{s} (l)}}{a_{k^{-}_{s} (l)}} \leqslant -\varepsilon - \frac{ES_{k^{-}_{s} (l)}}{a_{k^{-}_{s} (l)}} + \frac{S_{k^{-}_{s} (l)}}{a_{n}} \leqslant \frac{S_{k^{-}_{s} (l)}}{a_{n}} - \frac{ES_{n}}{a_{n}} \leqslant \\ \leqslant \frac{S_{n} - ES_{n}}{a_{n}} \leqslant \frac{S_{k^{+}_{s} (l)}}{a_{n}} - \frac{ES_{n}}{a_{n}} \leqslant \frac{S_{k^{+}_{s} (l)}}{a_{n}} - \frac{ES_{k^{+}_{s} (l)}}{a_{k^{+}_{s} (l)}} + \varepsilon \leqslant \alpha \frac{S_{k^{+}_{s} (l)}}{a_{k^{+}_{s} (l)}} - \frac{ES_{k^{+}_{s} (l)}}{a_{k^{+}_{s} (l)}} + \varepsilon = \\ = \alpha \frac{(S_{k^{+}_{s} (l)} - ES_{k^{+}_{s} (l)})}{a_{k^{+}_{s} (l)}} + (\alpha - 1) \frac{ES_{k^{+}_{s} (l)}}{a_{k^{+}_{s} (l)}} + \varepsilon \leqslant \alpha \frac{(S_{k^{+}_{s} (l)} - ES_{k^{+}_{s} (l)})}{a_{k^{+}_{s} (l)}} + (\alpha - 1) A + \varepsilon.
\end{multline*}

\noindent Thus, using~\eqref{e307}, we obtain

\begin{equation}\label{e311}
-\varepsilon - \left( 1 - \frac{1}{\alpha} \right) A \leqslant \liminf_{n \to \infty} \frac{S_{n} - ES_{n}}{a_{n}} \leqslant \limsup_{n \to \infty} \frac{S_{n} - ES_{n}}{a_{n}} \leqslant (\alpha - 1) A + \varepsilon
\end{equation}

\noindent almost surely. Since~\eqref{e311} is true for any $\alpha > 1$ and $\varepsilon > 0$, we get relation~\eqref{e203}.
\end{proof}

\begin{proof}[Proof of Theorem~\ref{T204}]

Without loss of generality it can be assumed that $EX_{n}=0$ for all $n \geqslant 1$. By Chebyshev's inequality, using~\eqref{e211} and Lemma~\ref{Lem301}, for any $\varepsilon >0$, we get

\begin{equation*}\label{e312}
\sum_{n=1}^{\infty} P \left( \left| \frac{S_{2^{n}}}{a_{2^{n}}} \right| > \varepsilon  \right) \leqslant \frac{1}{\varepsilon^{2}} \sum_{n=1}^{\infty} \frac{ES^{2}_{2^{n}}}{a_{2^{n}}^{2}} \leqslant C \varepsilon^{-2} \sum_{n=1}^{\infty} \frac{1}{ \psi(2^{n})} < \infty.
\end{equation*}

\noindent The application of Borel--Cantelli lemma yields to

\begin{equation*}\label{e313}
      \frac{S_{2^{n}}}{a_{2^{n}}} \to 0 \qquad \mbox{a.s.}
\end{equation*}

To complete the proof it is sufficiently to show that

\begin{equation*}\label{e314}
\lim_{n \to \infty} \max_{2^{n}<k \leqslant 2^{n+1}} \left| \frac{S_{k}}{a_{k}} \right| =0 \qquad \mbox{a.s.}
\end{equation*}

\noindent We have

\begin{multline}\label{e315}
\max_{2^{n}<k \leqslant 2^{n+1}} \left| \frac{S_{k}}{a_{k}} \right| = \max_{2^{n}<k \leqslant 2^{n+1}} \left| \frac{S_{k}-S_{2^{n}}+S_{2^{n}}}{a_{k}} \right| \leqslant \\ \leqslant \left| \frac{S_{2^{n}}}{a_{2^{n}}} \right| + \max_{1 \leqslant k \leqslant 2^{n}} \left| \frac{\sum_{i=2^{n}+1}^{2^{n}+k} X_{i}}{a_{2^{n+1}}} \right| \frac{a_{2^{n+1}}}{a_{2^{n}}}
\end{multline}

\noindent The first summand in the right-hand side of~\eqref{e315} convergences to zero almost surely. Taking into account assumption~\eqref{e210}, it is sufficiently to prove that

\begin{equation}\label{e316}
 \lim_{n \to \infty} \max_{1 \leqslant k \leqslant 2^{n}} \left| \frac{\sum_{i=2^{n}+1}^{2^{n}+k} X_{i}}{a_{2^{n+1}}} \right| =0 \qquad \mbox{a.s.}
\end{equation}

By Kolmogorov's inequality (see~\cite{Kolm74}), for any $\varepsilon >0$, we have

\begin{equation*}\label{e317}
\sum_{n=1}^{\infty} P \left( \max_{1 \leqslant k \leqslant 2^{n}} \left| \frac{\sum_{i=2^{n}+1}^{2^{n}+k} X_{i}}{a_{2^{n+1}}} \right| > \varepsilon \right) \leqslant \frac{1}{\varepsilon^{2}} \sum_{n=1}^{\infty} \frac{\sum_{i=2^{n}+1}^{2^{n+1} } EX_{i}^{2}}{a_{2^{n+1}}^{2}} \leqslant C \varepsilon^{-2} \sum_{n=2}^{\infty} \frac{1}{\psi(2^{n})} < \infty.
\end{equation*}

\noindent Thus,~\eqref{e316} follows from Borel--Cantelli lemma.
\end{proof}

{\sloppy
\begin{proof}[Proof of Theorem~\ref{T205}]
Suppose that conditions of Theorem~\ref{T205} are satisfied for sequences $\{a_{n}\}_{n=1}^{\infty}$, $\{b_{n}\}_{n=1}^{\infty}$ of non-negative numbers, nevertheless the series $\sum_{n=1}^{\infty} b_{n}/a_{n}^{2}$ diverges. Then there is a sequence of independent random variables $\{X_{n}\}_{n=1}^{\infty}$ such that $EX_{n} = 0$, $Var(X_{n}) = b_{n}$ for all $n \geqslant 1$, but relation~\eqref{e203} does not hold (see,~for example,~\cite{Petr75}). The sequence $\{X_{n}\}_{n=1}^{\infty}$ satisfies the conditions of Theorem~\ref{T204}, so~\eqref{e203} has to hold. This contradiction concludes the proof.
\end{proof}

} 

The next example shows that assumption~\eqref{e210} in Theorem~\ref{T204} cannot be dropped.

\medskip

\noindent \textbf{Example 1.} Let $a_{n} = 2^{n/2}$, $n \geqslant 1$. We consider the sequence of independent random variables $\{X_{n}\}_{n=1}^{\infty}$ such that

\begin{equation*}\label{e318}
P(X_{n} = 1) = P(X_{n} = -1) = \frac{1}{2} \qquad \mbox{ for } n \; = \; 1 \mbox{ or } 2,
\end{equation*}

\noindent and

\begin{equation*}\label{e319}
P(X_{n} = 2^{n/2}) = P(X_{n} = -2^{n/2}) = \frac{n-2}{4 n (n-1)},
\end{equation*}

\begin{equation*}\label{e320}
P(X_{n} = 0) = 1 - \frac{n-2}{2 n (n-1)}
\end{equation*}

\noindent for all $n \geqslant 3$. Then $EX_{n}=0$ for all $n \geqslant 1$, $Var(X_{1}) = Var(X_{2}) = 1$ and

\begin{equation*}\label{e321}
Var(X_{n}) = \frac{2^{n}}{n} - \frac{2^{n-1}}{n-1} \qquad \mbox{for all } n \geqslant 3.
\end{equation*}

\noindent We have

\begin{equation*}\label{e322}
Var(S_{n}) = \sum_{k=1}^{n} Var(X_{k}) = \frac{2^{n}}{n} \qquad \mbox{for all } n \geqslant 3.
\end{equation*}

\noindent Thus, the sequence of random variables $\{X_{n}\}_{n=1}^{\infty}$ satisfies condition~\eqref{e211} with ${a_{n} = 2^{n/2}}$, $n \geqslant 1$ and function $\psi (x) = x$, $x > 0$ (belonging to~$\Psi_{c}$). Moreover
{\sloppy

}

\begin{equation*}\label{e323}
\sum_{n=3}^{\infty} P(|X_{n}| = a_{n}) = \sum_{n=3}^{\infty} \frac{n-2}{2 n (n-1)} = \infty.
\end{equation*}

\noindent Application of Borel--Cantelli lemma yields to

\begin{equation}\label{e324}
 P(|X_{n}| = a_{n} \; \; \mbox{i.o.}) = 1.
\end{equation}

\noindent We shall suppose that relation~\eqref{e203} holds. Then we have

\begin{equation*}\label{e325}
 \frac{X_{n}}{a_{n}} = \frac{S_{n}}{a_{n}} - \frac{a_{n-1}}{a_{n}} \cdot \frac{S_{n-1}}{a_{n-1}} \to 0 \qquad \mbox{a.s.,}
\end{equation*}

\noindent which contradicts~\eqref{e324}.

\end{document}